\documentclass[11pt]{amsart}
\usepackage{amsmath, amsthm, mathabx,amscd, amsfonts, amssymb, hyperref, mathrsfs, textcomp, epsfig, a4wide,graphicx, IEEEtrantools, tikz, verbatim, xypic,enumerate}

\usetikzlibrary{matrix}

\newtheorem*{thm}{Theorem}

\newtheorem*{cor}{Corollary}

\theoremstyle{remark}
\newtheorem*{remark}{Remark}

     \RequirePackage{rotating}                   
    \def\HSt{%
       \setbox0=\hbox{$\widehat{\mathit{HS}}$}
       \setbox1=\hbox{$\mathit{HS}$}
       \dimen0=1.1\ht0
       \advance\dimen0 by 1.17\ht1
       \smash{\mskip2mu\raise\dimen0\rlap{%
          \begin{turn}{180}
              {$\widehat{\phantom{\mathit{HS}}}$}
           \end{turn}} \mskip-2mu    
                \mathit{HS}
    }{\vphantom{\widehat{\mathit{HS}}}}{}}

     \RequirePackage{rotating}                   
    \def\HMt{%
       \setbox0=\hbox{$\widehat{\mathit{HM}}$}
       \setbox1=\hbox{$\mathit{HM}$}
       \dimen0=1.1\ht0
       \advance\dimen0 by 1.17\ht1
       \smash{\mskip2mu\raise\dimen0\rlap{%
          \begin{turn}{180}
              {$\widehat{\phantom{\mathit{HM}}}$}
           \end{turn}} \mskip-2mu    
                \mathit{HM}
    }{\vphantom{\widehat{\mathit{HM}}}}{}}

\newcommand{\HMr}{{\mathit{HM}}}
\newcommand{\HMtil}{\widetilde{\mathit{HM}}}
\newcommand{\spin}{\mathfrak{s}}
\newcommand{\ztwo}{\mathbb{F}}
\newcommand{\Pin}{\mathrm{Pin}(2)}

\begin{document}

\title{A remark on taut foliations and Floer homology}

\author{Francesco Lin}
\address{Department of Mathematics, Columbia University} 
\email{flin@math.columbia.edu}

\begin{abstract}It is well-known that the reduced Floer homology of a rational homology sphere admitting a taut foliation does not vanish. We strengthen this by showing that (when thought of as an $\ztwo[U]$-module) it also admits a direct $\ztwo$-summand. Hence, one could potentially detect counterexamples to the foliation part of the $L$-space conjecture via purely Floer theoretic computations.
\end{abstract}

\maketitle

We will formulate our discussion using the monopole Floer homology groups, but our results apply to Heegaard Floer homology as the theories are isomorphic\footnote{The dictionary is $\HMtil\cong\widehat{HF}$, $\HMt\cong HF^+$ and $\HMr\cong HF^{red}$.} (see \cite{KLT}, \cite{CGH} and subsequent papers). We will work $\ztwo=\mathbb{Z}/2\mathbb{Z}$ coefficients, and foliations and contact structures will always be cooriented.
\par
It is well-known that if a rational homology sphere $Y$ admits a taut foliation, then it is irreducible and
\begin{equation}\label{ineq}
\mathrm{dim}_{\ztwo} \HMtil_*(Y)>|H_1(Y;\mathbb{Z})|,
\end{equation}
see \cite{KMOS}. Part of the influential $L$-\textit{space conjecture} predicts that the converse holds \cite{Juh}. This has been verified in many classes of examples, including all graph manifolds \cite{HRRW} and large families of hyperbolic ones (see for example \cite{Dun}, \cite{Kri}, \cite{San}).
\par
Notice that the inequality (\ref{ineq}) is equivalent to the condition 
\begin{equation}\label{nonzero}
\HMr_*(Y)\neq 0,
\end{equation}
i.e. the reduced Floer group is non-vanishing. Unlike $\HMtil_*(Y)$, the reduced group $\HMr_*(Y)$ also admits the structure of finitely generated graded $\ztwo[U]$-module (where $U$ has degree $-2$). Our main result is that the existence of a taut foliation implies restrictions on such module structure which strengthen (\ref{nonzero}). 

\begin{thm}Suppose the rational homology sphere $Y$ admits a taut foliation. Then the $\ztwo[U]$-module $\HMr_*(Y)$ admits a direct $\ztwo$ summand.
\end{thm}

\begin{cor}
If the statement in the $L$-space conjecture regarding taut foliations holds, then for any rational homology sphere $Y$, the $\ztwo[U]$-module $\HMr_*(Y)$ either vanishes or admits a direct $\ztwo$-summand.\end{cor}
\begin{proof}[Proof of Corollary]
If $Y$ is irreducible, this directly follows from the Theorem. In general, the connected sum formula of Bloom-Mrowka-Ozsv\'ath (see \cite{Lin} for an exposition) shows that if $Y$ is such that $\HMr_*(Y)$ has no $\ztwo$-summands, then the same holds for all of its prime summands.
\end{proof}

In particular, the Corollary implies that one could potentially detect counterexamples to the $L$-space conjecture via purely Floer theoretic computations.
\begin{remark}
The author is not aware of heuristics in Seiberg-Witten theory possibly suggesting that for every rational homology sphere $Y$ the module $\HMr_*(Y)$ either vanishes or admits a direct $\ztwo$-summand.
\end{remark}

\begin{proof}[Proof of Theorem]
This follows from refining the usual proof of non-vanishing; we will adopt the strategy discussed in Chapter $41$ of \cite{KM} using TQFT properties of monopole Floer homology, which we briefly recall.
\par
The works \cite{ET}, \cite{Bow}, \cite{KR} imply that the taut foliation can be perturbed in a $C^0$-small way into two contact structures $\xi_{\pm}$, respectively positive and negative, such that the cobordism $I\times Y$ is a weak symplectic filling of its boundary $(-Y,\xi_-)\coprod(Y,\xi_+)$. Because $Y$ is a rational homology sphere, this can be modified near the boundary components to be a strong symplectic filling \cite{OO}. We can then attach strong symplectic caps $(W_{\pm},\omega_{\pm})$ of $(\pm Y,\xi_{\pm})$ \cite{EK} (for simplicity, we can also assume that both have $b^+\geq 2$) to obtain a closed symplectic $4$-manifold $(X,\omega)$. Notice that $W_{\pm}$ have boundary $\mp Y$ respectively.
\par
Now, Taubes' theorem \cite[Section $41.2$]{KM} implies that the Seiberg-Witten invariant $\mathfrak{m}(X,\spin_{\omega})$ is $1\in\ztwo$, where $\spin_\omega$ denotes the canonical spin$^c$ structure. Because in our decomposition both pieces have $b^+\geq 1$, we can also compute this value using the pairing theorem\footnote{We remark that a spin$^c$ structure on $X$ is determined by its restrictions to $W_{\pm}$ because $Y$ is a rational homology sphere, so that we do not use local coefficients as in the general proof in \cite[Chapter $41$]{KM}.}
\begin{equation}\label{pairing}
\langle \psi_{(W_-,\spin_{\omega_-})},\psi_{(W_+,\spin_{\omega_+})}\rangle=\mathfrak{m}(X,\spin_{\omega})=1\in\ztwo,
\end{equation}
where we consider the relative invariants
\begin{equation*}
\psi_{(W_{\pm},\spin_{\omega_{\pm}})}\in \HMr_*(\mp Y)
\end{equation*}
of the cobordisms, and the perfect pairing 
\begin{equation}\label{pair}
\langle\cdot,\cdot\rangle: \HMr_*(Y)\otimes \HMr_*(-Y)\rightarrow \ztwo
\end{equation}
induced by Poincar\'e duality $\HMr^*(Y)\equiv\HMr_*(-Y)$, see \cite[(3.22)]{KM}. 
\par
This directly implies the usual statement that $\HMr_*(Y)\neq0$. To refine this, we consider the classes
\begin{equation*}
c_{\pm}:=\textbf{c}(\xi_{\pm})\in \HMr_*(\mp Y)
\end{equation*}
where (by abuse of notation) we denote by $\textbf{c}(\xi_{\pm})$ the image of the contact invariant (which is an element in $\HMt_*(\mp Y)$ \cite{KMOS}) in reduced Floer homology. The work \cite{Ech} on naturality for strong symplectic cobordisms identifies the contact invariant of a contact three-manifold with the relative invariant of any strong symplectic cap. In our situation, this implies
\begin{equation*}
c_{\pm}=\psi_{(W_\pm,\spin_{\omega_\pm})}
\end{equation*}
so that $c_{\pm}$ have non-zero pairing under (\ref{pair}). As the contact invariant $\textbf{c}$ always belongs to the kernel of $U$ (see \cite{Mun}), we have that $U\cdot{c}_{\pm}=0\in \HMr_*(\mp Y)$. With this in mind, we decompose $\HMr_*(Y)$ as a direct sum of cyclic $\ztwo[U]$-modules:
\begin{equation}\label{HMpl}
\HMr_*(Y)\cong \bigoplus_i \ztwo[U]/U^{k_i},
\end{equation}
where each summand implicitly has some grading shift. This naturally induces a decomposition of the dual $\ztwo[U]$-module $\HMr_*(-Y)$: we identify as vector spaces
\begin{equation}\label{HMmin}
\HMr_*(-Y)\equiv \bigoplus_i \ztwo[U]/U^{k_i},
\end{equation}
and reverse the direction of the $U$-action; let us denote it by $U_*$ to prevent confusion. Because $U\cdot c_-=0$, the non-zero components of $c_-$ all lie at the bottom of the $U$-towers in (\ref{HMpl}), i.e. are of the form 
 \begin{equation*}
 U^{k_i-1}\in \ztwo[U]/U^{k_i}.
 \end{equation*}
Similarly, because $U_*\cdot c_+=0$, the non-zero components of $c_+$ all lie at the bottom of the $U_*$-towers, which correspond to the elements $1\in \ztwo[U]/U^{k_j}$ in (\ref{HMmin}). Because $\langle c_-,c_+\rangle\neq0$, $c_-$ and $c_+$ share a non-zero component in the same cyclic $\ztwo[U]$-summand, which is therefore forced to be a copy of $\ztwo$.
\end{proof}

\begin{remark}
As pointed out by Baldwin and Sivek, the proof also provides concrete obstructions for a contact structure $\xi$ to obtained (up to isotopy) from a taut foliation via the Eliashberg-Thurston perturbation (in terms of the interplay between $U$ and the contact invariant of $\xi$). For instance, these can be applied to the examples constructed in \cite{Kar}.
\end{remark}

\vspace{0.3cm}

Of course, at the practical level the determination of the $\ztwo[U]$-module structure on $\HMr_*(Y)$ is significantly more complicated than the already challenging computation of $\mathrm{dim}_{\ztwo} \HMtil_*(Y)$. However, we point out how the simplest case in which $\HMr_*(Y)$
is a cyclic $\ztwo[U]$-module, i.e.
\begin{equation*}
\HMr_*(Y)\cong \ztwo[U]/U^k \text{ for some }k\geq1,
\end{equation*}
can be approached directly by looking at the spin$^c$ decomposition and relative gradings in $\HMtil_*(Y)$. Notice that the latter are both explicitly computable in concrete examples via the \texttt{bfh\_python} package \cite{BZ}.
\par
Because $\HMtil_*$ is the mapping cone of the action of $U$ on $\HMt_*$, we have that $\HMr_*(Y)\cong \ztwo[U]/U^k$ for some $k\geq1$ if and only if
\begin{equation*}
\mathrm{dim}_{\ztwo}\HMtil_*(Y)=|H_1(Y;\mathbb{Z})|+2,
\end{equation*}
i.e. $Y$ is an \textit{almost $L$-space} in the sense of \cite{BS} (see \cite{Binn} for related general results). The Corollary then implies that, assuming the $L$-space conjecture, for any almost $L$-space $\HMr_*(Y)=\ztwo$. Furthermore, by conjugation symmetry, the latter is necessarily supported in a spin structure $\spin_0$. Then the main result of \cite{HKL} (which uses $\Pin$-symmetry) implies that this copy of $\ztwo$ is either in the same grading as the bottom of the $U$-tower in $\HMt_*$ or one below it. We record these observations as follows.
\begin{cor}
Assuming the statement in the $L$-space conjecture regarding taut foliations holds. Then any almost $L$-space $Y$ has a spin structure $\spin_0$ such that
\begin{equation*}
\HMtil_*(Y,\spin)
=\begin{cases}
\ztwo^3\text{ if }\spin=\spin_0\\
\ztwo \text{ otherwise}
\end{cases}
\end{equation*}
and furthermore $\HMtil_*(Y,\spin_0)$ is concentrated in exactly two consecutive gradings.
\end{cor}

It would be very interesting if one could prove that the conclusion of this result holds (even just in large families) without assuming the $L$-space conjecture. Dunfield has informed the author that computations with \texttt{bfh\_python} confirm this condition holds for 2,770 hyperbolic almost $L$-spaces which arise as double branched covers of non-alternating links in $S^3$ with at most 15 crossings. Finally, we remark that severe restrictions on the topology of \textit{Stein} fillings for the class of manifolds arising in the Corollary are proved in \cite{LinStein}; it is natural to wonder whether there are extensions of such results to other kinds of fillability properties more closely related to taut foliations.
\vspace{0.5cm}

\textit{Acknowledgements. }The author thanks Antonio Alfieri, John Baldwin and Nathan Dunfield for some helpful conversations. This work was partially supported by the Alfred P. Sloan Foundation and NSF grant DMS-2203498.

\vspace{0.3cm}

\bibliographystyle{alpha}
\bibliography{biblio}

\begin{thebibliography}{HRRW20}

\bibitem[Bin23]{Binn}
Fraser Binns.
\newblock The ${CFK}^{\infty}$ type of almost {L}-space knots.
\newblock {\em preprint}, 2023.

\bibitem[Bow16]{Bow}
Jonathan Bowden.
\newblock Approximating {$C^0$}-foliations by contact structures.
\newblock {\em Geom. Funct. Anal.}, 26(5):1255--1296, 2016.

\bibitem[BS22]{BS}
John Baldwin and Steven Sivek.
\newblock Characterizing slopes for $5_2$.
\newblock {\em preprint}, 2022.

\bibitem[CGH11]{CGH}
Vincent Colin, Paolo Ghiggini, and Ko~Honda.
\newblock Equivalence of {H}eegaard {F}loer homology and embedded contact
  homology via open book decompositions.
\newblock {\em Proc. Natl. Acad. Sci. USA}, 108(20):8100--8105, 2011.

\bibitem[Dun20]{Dun}
Nathan~M. Dunfield.
\newblock Floer homology, group orderability, and taut foliations of hyperbolic
  3-manifolds.
\newblock {\em Geom. Topol.}, 24(4):2075--2125, 2020.

\bibitem[Ech20]{Ech}
Mariano Echeverria.
\newblock Naturality of the contact invariant in monopole {F}loer homology
  under strong symplectic cobordisms.
\newblock {\em Algebr. Geom. Topol.}, 20(4):1795--1875, 2020.

\bibitem[EH02]{EK}
John~B. Etnyre and Ko~Honda.
\newblock On symplectic cobordisms.
\newblock {\em Math. Ann.}, 323(1):31--39, 2002.

\bibitem[ET98]{ET}
Yakov~M. Eliashberg and William~P. Thurston.
\newblock {\em Confoliations}, volume~13 of {\em University Lecture Series}.
\newblock American Mathematical Society, Providence, RI, 1998.

\bibitem[HKL19]{HKL}
Jonathan Hanselman, \c{C}a\u{g}atay Kutluhan, and Tye Lidman.
\newblock A remark on the geography problem in {H}eegaard {F}loer homology.
\newblock In {\em Breadth in contemporary topology}, volume 102 of {\em Proc.
  Sympos. Pure Math.}, pages 103--111. Amer. Math. Soc., Providence, RI, 2019.

\bibitem[HRRW20]{HRRW}
Jonathan Hanselman, Jacob Rasmussen, Sarah~Dean Rasmussen, and Liam Watson.
\newblock L-spaces, taut foliations, and graph manifolds.
\newblock {\em Compos. Math.}, 156(3):604--612, 2020.

\bibitem[Juh15]{Juh}
Andr\'{a}s Juh\'{a}sz.
\newblock A survey of {H}eegaard {F}loer homology.
\newblock In {\em New ideas in low dimensional topology}, volume~56 of {\em
  Ser. Knots Everything}, pages 237--296. World Sci. Publ., Hackensack, NJ,
  2015.

\bibitem[Kar14]{Kar}
\c{C}a\u{g}r\i Karakurt.
\newblock Contact structures on plumbed 3-manifolds.
\newblock {\em Kyoto J. Math.}, 54(2):271--294, 2014.

\bibitem[KLT20]{KLT}
\c{C}a\u{g}atay Kutluhan, Yi-Jen Lee, and Clifford~Henry Taubes.
\newblock {${\rm HF}={\rm HM}$}, {I}: {H}eegaard {F}loer homology and
  {S}eiberg-{W}itten {F}loer homology.
\newblock {\em Geom. Topol.}, 24(6):2829--2854, 2020.

\bibitem[KM07]{KM}
Peter Kronheimer and Tomasz Mrowka.
\newblock {\em Monopoles and three-manifolds}, volume~10 of {\em New
  Mathematical Monographs}.
\newblock Cambridge University Press, Cambridge, 2007.

\bibitem[KMOS07]{KMOS}
P.~Kronheimer, T.~Mrowka, P.~Ozsv\'{a}th, and Z.~Szab\'{o}.
\newblock Monopoles and lens space surgeries.
\newblock {\em Ann. of Math. (2)}, 165(2):457--546, 2007.

\bibitem[KR17]{KR}
William~H. Kazez and Rachel Roberts.
\newblock {$C^0$} approximations of foliations.
\newblock {\em Geom. Topol.}, 21(6):3601--3657, 2017.

\bibitem[Kri20]{Kri}
Siddhi Krishna.
\newblock Taut foliations, positive 3-braids, and the {L}-space conjecture.
\newblock {\em J. Topol.}, 13(3):1003--1033, 2020.

\bibitem[Lin17]{Lin}
Francesco Lin.
\newblock {${\rm Pin}(2)$}-monopole {F}loer homology, higher compositions and
  connected sums.
\newblock {\em J. Topol.}, 10(4):921--969, 2017.

\bibitem[Lin20]{LinStein}
Francesco Lin.
\newblock Indefinite {S}tein fillings and {${\rm PIN}(2)$}-monopole {F}loer
  homology.
\newblock {\em Selecta Math. (N.S.)}, 26(2):Paper No. 18, 15, 2020.

\bibitem[MnE21]{Mun}
Juan Mu\~{n}oz Ech\'aniz.
\newblock A monopole invariant for families of contact structures.
\newblock {\em preprint}, 2021.

\bibitem[OO99]{OO}
Hiroshi Ohta and Kaoru Ono.
\newblock Simple singularities and topology of symplectically filling
  {$4$}-manifold.
\newblock {\em Comment. Math. Helv.}, 74(4):575--590, 1999.

\bibitem[San22]{San}
Diego Santoro.
\newblock L-spaces, taut foliations and the {W}hitehead link.
\newblock {\em to appear in Algebraic \& Geometric Topology}, 2022.

\bibitem[Zha]{BZ}
Bohua Zhan.
\newblock bfh\_python.
\newblock {\em \url{https://pages.uoregon.edu/lipshitz/bfh/}}.

\end{thebibliography}

\vspace{0.3cm}

\end{document}